\documentclass[reqno,twoside,12pt]{amsart}

\hoffset - 2.2cm

\usepackage{graphicx}
\usepackage{amstext}
\usepackage{pstricks,pst-node,pst-text,pst-3d}
\usepackage{amsmath}
\usepackage{amsthm}
\usepackage{amssymb}

\usepackage{enumerate}
\usepackage[polish, german, english]{babel}

\theoremstyle{definition}
\newtheorem{Definition}{Definition}[section]
\newtheorem{pr}{Problem}
\theoremstyle{plain}
\newtheorem{Theorem}{Theorem}[section]
\newtheorem{Lemma}[Theorem]{Lemma}
\newtheorem{Remark}[Theorem]{Remark}

\newcommand{\R}{\mathbb{R}}

\newcommand{\Z}{\mathbb{Z}}
\newcommand{\V}{\mathbb{V}}
\newcommand{\W}{\mathbb{W}}

\newcommand{\N}{\mathbb{N}}

\newcommand{\F}{\mathcal{F}_{*}}
\newcommand{\I}{\mathbb{I}}

\newcommand{\Id}{\operatorname{Id}}

\newcommand{\SO}{\operatorname{SO}}

\renewcommand{\H}{\mathcal{H}}
\newcommand{\sub}{\overline{\operatorname{sub}}}
\newcommand{\im}{\operatorname{im}}

\newcommand{\A}{\mathcal{A}}

\newcommand{\z}{\mathcal{Z}}
\newcommand{\K}{\mathcal{K}}

\begin{document}

\title[Symmetry breaking of solutions of non-cooperative elliptic systems]{Symmetry breaking of solutions \\ of non-cooperative elliptic systems}

\author{Piotr Stefaniak}
\address{Faculty of Mathematics and Computer Science\\
Nicolaus Copernicus University \\
PL-87-100 Toru\'{n} \\ ul. Chopina $12 \slash 18$ \\
Poland}

\email{cstefan@mat.umk.pl}
\date{\today}

\keywords{symmetry breaking, non-cooperative elliptic system, equivariant degree}
\subjclass[2010]{Primary: 35J57; Secondary: 35B06.}
\thanks{Partially supported by the National Science Centre,  Poland,  under grant    DEC-2012/05/B/ST1/02165}

\begin{abstract}
In this article we study the symmetry breaking phenomenon of solutions of non-cooperative elliptic systems. We apply the degree for $G$-invariant strongly indefinite functionals to obtain simultaneously a symmetry breaking and a global bifurcation phenomenon.
\end{abstract}

\maketitle
\section{Introduction}

In this paper, we consider a symmetry breaking of solutions of non-cooperative elliptic systems of the form:
\begin{equation}\label{problem1}
\left\{ \begin{array}{rcl}
-\Delta w_1 = \nabla_{w_1} F(w_1,w_2)+f_1& \text{in}&\Omega\\
\Delta w_2=\nabla_{w_2} F(w_1,w_2)+f_2& \text{in}&\Omega\\
\frac{\partial w_1}{\partial \nu}=\frac{\partial w_2}{\partial \nu}=0 & \text{on}& \partial \Omega,
\end{array}\right.
\end{equation}
where $\R^n$ is an orthogonal representation of a compact Lie group $G$, $\Omega\subset\R^n$ is an open, bounded, $G$-invariant set with a smooth boundary and $F\in C^2(\R^2,\R)$. That is we discuss the existence of a $G$-symmetric function $(f_1,f_2)$ such that there is a $K$-symmetric solution $(w_1,w_2)$ of system \eqref{problem1}, where $K$ is a closed subgroup of $G$. If such a solution exists, we say that occurs a symmetry breaking of solutions of problem \eqref{problem1}.

The problem of symmetry breaking has been studied by many authors under various assumptions on $F$ and $\Omega$, see for instance \cite{Budd}-\cite{Dancer1}, \cite{Srikanth1, Jager, Lauterbach}, \cite{Srikanth}-\cite{Srikanth2}, \cite{Srikanth4}. Of course this list is far from being complete. The authors have used different tools to obtain their results: Rybakowski's homotopy index, the equivariant Conley index or the Leray-Schauder degree. We have applied the degree for  $G$-invariant strongly indefinite functionals, see \cite{degree}, to obtain our results. Using this degree we have formulated conditions on $F$ which enable us to decide whether there is a connected set of solutions of the main problem. 

The idea of the proof of our main result is to reduce the problem to a bifurcation one.  We follow the idea from \cite{Dancer}, due to Dancer. The author has used a different tool, that is Rybakowski's homotopy index, see \cite{Rybakwoski}, which cannot be used to prove our results, because the functional corresponding to system \eqref{problem1} is strongly indefinite.  Moreover, using Rybakowski's or Conley indices it is only possible to obtain a sequence of solutions of the symmetry breaking problem. Using the degree for  $G$-invariant strongly indefinite functionals we have obtained a global bifurcation of solutions that problem. Moreover, our method can be used to handle a number of related problems.

After this introduction our article is organised as follows.

In section 2 we introduce our notation and reduce the symmetry breaking problem to a bifurcation problem.

In section 3 we consider a system of elliptic equations and recall basic properties of the operator induced by this system. We formulate the symmetry breaking and the corresponding bifurcation problem for this system. We calculate the degree for  $N(K)$-invariant strongly indefinite functionals for an operator associated with a linear system of equations, where $N(K)$ is the normalizer of a subgroup $K$ of $G$. We use this results to proceed some computations in a nonlinear case. 

In section 4 we formulate and prove the main results of this article. To do it we use the abstract results from the previous sections. 

In section 5 we illustrate our method.

To make this article self-contained, we have included in section 6 the definition of the Euler ring $U(G)$ of a compact Lie group $G$ and the definition and basic properties of the degree for  $G$-invariant strongly indefinite functionals, due to Go{\l}\c{e}biewska and Rybicki, see \cite{degree}.

\section{Preliminaries}

 Throughout this article $G$ stands for a compact Lie group and $\sub(G)$ for the set of closed subgroups of $G$. Let $(\H,\langle \cdot,\cdot\rangle)$  be a separable Hilbert space, which is an orthogonal representation of $G$ and let
 $\H^K = \{x\in \H: \forall_{g\in K}\ gx=x\}$ be the set of all fixed points of the action of a subgroup $K\in\sub(G)$.
 The set $N(K)$ is the normalizer of a subgroup $K\in\sub(G)$, i.e. $N(K)=\{g\in G: gK=Kg\}$.
 Fix $k\in\N$. Let $C^k_G(\H,\R)$ denote the set of all $G$-invariant functionals of class $C^k$, i.~e. $\Psi (g x)=\Psi (x)$, where $\Psi \in C^k_G(\H,\R),\ g\in G$ $x\in\H$, and $C^{k-1}_G(\H,\H)$ the set of all $G$-equivariant operators of class $C^{k-1}$, i.~e. $T(gx)=gT(x)$, where $T \in C^{k-1}_G(\H,\H),\ g\in G$, $x\in \H$.
 It can be easily shown that for a fixed $K\in\sub(G)$, $\H^G\subset\H^K$ and if $\Psi \in C^k_G(\H,\R)$, then the gradient $\nabla\Psi \in C^{k-1}_G(\H,\H)$, $k\in\N$.
 We denote by $B_{\gamma}(\H,p)$ the open unit ball in $\H$ centered at a point $p$ of radius $\gamma$. Moreover, we  put $B(\H, p)=B_{1}(\H,p)$, $B_{\gamma}(\H)=B_{\gamma}(\H,0)$ and $B(\H)=B_{1}(\H,0)$.
 Suppose that $\Lambda$ is a linear space of parameters, $\Psi \in C^k_G(\H\times\Lambda,\R)$ is such that  $\nabla_u \Psi(0,\lambda)=0$ for every $\lambda\in \Lambda$. Consider the equation
 \begin{equation}\label{bifogol}
 \nabla_u \Psi(u,\lambda)=0.
 \end{equation}
 Define a set of non-zero solutions of \eqref{bifogol} by
 $\mathcal{N}=\{(u,\lambda)\in (\H\setminus\{0\})\times\Lambda: \nabla_u \Psi(u,\lambda)=0\}$,
 fix $\lambda_0\in\Lambda$ and denote by $C(\lambda_0)$ a connected component of the closure $\operatorname{cl}(\mathcal{N})$ such that $(0,\lambda_0)\in C(\lambda_0)$.

 \begin{Definition}
 A point $(0,\lambda_0)\in \{0\}\times\Lambda$  is said to be a local bifurcation point of solutions
of equation \eqref{bifogol}, if $(0, ƒ\lambda_0)\in \operatorname{cl}(\mathcal{N})$.
A point $(0, ƒ\lambda_0)\in \{0\}\times\Lambda$ is said to be a branching point of non-zero solutions of equation \eqref{bifogol}, if
$C(\lambda_0) \neq \{(0, \lambda_0)\}$.
A point $(0, ƒ\lambda_0)\in \{0\}\times\Lambda$ is said to be a global bifurcation point of non-zero solutions of equation  \eqref{bifogol}, if either $C(\lambda_0) \cap ({0}\times(\Lambda \setminus \{\lambda_0\}) \neq \emptyset$ or $C(\lambda_0)$ is not bounded.
 \end{Definition}

\begin{pr}\label{pr1}
Let $T\in C^0_G(\H,\H)$. Does there exist $w\in \H^K\backslash \H^G$ such that $T(w)\in \H^G$?
\end{pr}

For subspaces $\H_2\subset \H_1\subset\H$ set $\H_1\ominus\H_2=\{u\in\H_1:\langle u,v \rangle=0\ \forall_{v\in \H_2}\}$.

Consider a $G$-equivariant projection $\pi\colon \H\to \H$ such that $\im \pi = (\H^G)^{\bot}$. Then
 $\im(I-\pi)=\H^G$ and note that $\pi(\H^K)\subset \H^K$ for every $K\in\sub(G)$.
Define  $\pi_1\colon \H^K\to \H$ to be the composition $\pi_1 =  \pi\circ i$, where
$i\colon \H^K\to (\H\ominus\H^K)\oplus\H^K$ is the embedding given by $i(x)=(0,x)$. The mapping $i$ is $N(K)$-equivariant (the space $\H^K$ is $N(K)$-invariant and does not have to be $G$-invariant), so $\pi_1$ is also  $N(K)$-equivariant. It is easy to verify that $\im\pi_1=\H^K\ominus\H^G$. Let $\H^K=\im \pi_1\oplus\Lambda$, where $\Lambda=\H^G$.

In \cite{Dancer} it has been shown that Problem \ref{pr1} is equivalent to the following
\begin{pr}\label{pr2}
Let $T\in C^0_G(\H,\H)$. Do there exist $\lambda \in \Lambda$ and $u\in \im \pi_1 \backslash\{0\}$ satisfying the equation
$
(\pi_1\circ T\circ i)(u,\lambda)=0?
$
\end{pr}

Define the operator $\A\in C^0_{N(K)}(\im \pi_1 \oplus \Lambda, \im \pi_1)$ by $\A(u,\lambda)=\pi_1(T(i(u,\lambda)))$.
It is easy to verify that the operator $\A$ is well defined.

The following remark follows from the definition of $\pi_1$ and the equality $\A(0,\lambda)=0$ for every $\lambda \in \Lambda$.

\begin{Remark}\label{symbrAbif}
If there exists a bifurcation point of solutions of the equation
$\A(u,\lambda)=0$, then the answer to Problem \ref{pr1} is affirmative.
\end{Remark}

In view of remark \ref{symbrAbif} our aim is to study the bifurcations of solutions of the equation $\A(u,\lambda)=0$.

Throughout the rest of this section we will need the following assumptions:
\begin{enumerate}
\item $\Phi\in C^2_G(\H,\R)$,
\item $\Phi(w) =\frac{1}{2}\langle L w, w \rangle-\eta(w)$,
\item $L\colon \H \to \H$ is a linear, bounded, self-adjoint, G-equivariant Fredholm operator of index~0,
\item $\nabla \eta\in C^1_G(\H,\H)$ is a completely continuous operator.

\end{enumerate}

From now on we put $T=\nabla\Phi$. Because the operator $L$ is $G$-equivariant, $L(\H^G)\subset \H^G$, $L(\H^K)\subset \H^K$. Since the operator L is self-adjoint, we obtain the following
\[\left.\begin{array}{cccc}
 &\H\ominus\H^K& & \H\ominus\H^K \\
 & \oplus& & \oplus\\
L\colon &\H^K\ominus\H^G&\to & \H^K\ominus\H^G, \\
 & \oplus& & \oplus\\
 &\H^G& & \H^G
\end{array}\right.
\ L=
\left[ \begin{array}{ccc}
 L_1&0&0 \\
 0&L_2&0 \\
  0&0&L_3
 \end{array} \right].
\]
From the above we get $\pi_1(L(i(u,\lambda)))=\pi_1(L(0,u,\lambda))=\pi_1(0,L_2u,L_3\lambda)=L_2u.$ Therefore $\A(u,\lambda)=\pi_1(\nabla_u\Phi(u,\lambda))=L_2u-\pi_1(\nabla\eta(i(u,\lambda))).$

\begin{Lemma}\label{Lemma1}
For every $\lambda \in \Lambda$ the operator $\A(\cdot,\lambda)\in C^1_{N(K)}(\im \pi_1, \im \pi_1)$ is gradient.
\end{Lemma}
We refer the reader to \cite{Dancer} for the proof of the above lemma.

From the above lemma it follows that the equation $\A(u,\lambda)=0$ has a variational and symmetric structure. To study bifurcations of solutions of the equation can be used the degree for  $N(K)$-invariant strongly indefinite functionals $\nabla_{N(K)}$-$\deg(\cdot,\cdot)$, which is an element of the Euler ring $U(N(K))$ of a compact Lie group $N(K)$, see Appendix for the definitions and basic properties.

\section{Elliptic system}
In this section we study the strongly indefinite functional associated with a system of elliptic equations.

Consider the following system
\begin{equation}\label{rowNeuSN}
\left\{ \begin{array}{rcl}
-\Delta w_1 = \nabla_{w_1} F(w_1,w_2)& \text{in}&\Omega\\
\Delta w_2=\nabla_{w_2} F(w_1,w_2)& \text{in}&\Omega\\
\frac{\partial w_1}{\partial \nu}=\frac{\partial w_2}{\partial \nu}=0 & \text{on}& \partial \Omega,
\end{array}\right.
\end{equation}
where
\begin{enumerate}
\item $\Omega$ is an open, bounded and $G$-invariant subset of an orthogonal $G$-representation $\R^n$, with a smooth boundary,
\item $F\in C^2(\R^2,\R)$,
\item  $|\nabla^2 F(y)|\leq a +b|y|^q$, where $a,\ b\in\R,\ q<\frac{4}{n-2}$ for $n\geq 3$ and $q<\infty$ for $n=2$.
\end{enumerate}

Put in the previous section $\H=H^1(\Omega)\oplus H^1(\Omega)$.
Since $H^1(\Omega)$ is an orthogonal $G$-representation with the action given by $(g,u)(x)\mapsto u(g^{-1}x)$ for  $g\in G,\ u\in H^1(\Omega),\ x\in\Omega$, so is $\H$, where $G$ acts on this space by $(g,(u,v))(x)\mapsto (u(g^{-1}x),v(g^{-1}x))$ for  $g\in G,\ u,\ v\in H^1(\Omega),\ x\in\Omega$.
Put  $L=
\left[\begin{array}{cc}
1 &0\\
0& -1
\end{array}\right]$. For brevity we use the same notation for a matrix and the operator $H^1(\Omega)\oplus H^1(\Omega)\to H^1(\Omega)\oplus H^1(\Omega)$ induced by the matrix.

Recall that a weak solution of the system is a function $w \in \H$  such that
\begin{equation*}
\forall_{v\in \H}~\int\limits_{\Omega} \langle L\nabla w(x), \nabla v(x)\rangle - \langle \nabla F(w(x)),v(x) \rangle dx= 0,
\end{equation*}
where $\langle\cdot,\cdot\rangle$ are the standard inner products in $\R^{2n}$ and $\R^2$.

Put in the previous section 
\begin{equation}\label{Phi}
\Phi(w)= \frac{1}{2} \int\limits_{\Omega}|\nabla w_1(x)|^2 - |\nabla w_2(x)|^2dx -\int\limits_{\Omega} F(w(x))dx
\end{equation}
\[= \frac{1}{2} \int\limits_{\Omega}|\nabla w_1(x)|^2 - |\nabla w_2(x)|^2 +|w_1(x)|^2 - |w_2(x)|^2dx+\]
\[- \int\limits_{\Omega}\frac{1}{2}|w_1(x)|^2 - \frac{1}{2}|w_2(x)|^2 + F(w(x))dx=\]
\[= \frac{1}{2} \int\limits_{\Omega}\langle  \nabla (Lw(x)),\nabla w(x)\rangle +\langle L w(x), w(x)\rangle dx-\eta(w)=
\frac{1}{2}\langle Lw,w \rangle_{\H}-\eta(w),\]
where
\[
\eta(w)=- \int\limits_{\Omega}\frac{1}{2}|w_1(x)|^2 - \frac{1}{2}|w_2(x)|^2 + F(w(x))dx
\]
and therefore
\begin{equation}\label{eta}
  \langle \nabla\eta(w),v \rangle_{\H}= \int\limits_{\Omega}\langle L w(x),v(x)\rangle + \langle\nabla F(w(x)),v(x)\rangle dx.
\end{equation}

Then
$\Phi(w)=
\frac{1}{2}\langle Lw,w \rangle_{\H}-\eta(w)$,  $\nabla \Phi (w) = L w -\nabla\eta(w)$ and $\nabla \eta$ is a completely continuous operator (and consequently compact).
Moreover, a function $w\in \H$ is a weak solution of system \eqref{rowNeuSN} if and only if $\nabla \Phi(w)=0$, that is  $w$ is a critical point of $\Phi$.

We study breaking of symmetries of critical orbits for the functional $\Phi$. To do this, fix $K\in \sub(G)$ and recall that we have defined an equivariant orthogonal projection
 $\pi_1 \colon \im \pi_1 \oplus \Lambda \to \im \pi_1$,
where $\im \pi_1= \H^K \ominus \H^{G}$ and $\Lambda=\H^{G}$.  We have also defined the operator  $\A\in C^1_{N(K)}(\im \pi_1
\oplus \Lambda, \im \pi_1)$ by $\A(u,\lambda)=\pi_1(\nabla\Phi(i(u,\lambda)))$, that is
\[
\A(u,\lambda)= \pi_1(\nabla
\Phi(i(u,\lambda)))= L_2 u - \pi_1(\nabla\eta(i(u,\lambda))),
\]
 where $i$ is an embedding  $\H^K$ in $(\H^K)^{\bot}\oplus\H^K$ defined by  $i(x)=(0,x)$. From lemma \ref{Lemma1}
it follows that the operator $\A(\cdot,\lambda)\in C^1_{N(K)}(\im \pi_1, \im
\pi_1)$ is gradient for every $\lambda\in \Lambda$.

 Denote by $\sigma(-\Delta,\Omega)=\{0=\mu_1<\mu_2<\ldots\}$ the set of eigenvalues of the elliptic equation on $\Omega$ with the Neumann boundary condition and $\V_{-\Delta}(\mu_k)$ the eigenspace associated with $\mu_k\in\sigma(-\Delta;\Omega)$. We also use the following notation
\begin{enumerate}
\item $\H^0=\{0\}$,
\item $\H_k=\V_{-\Delta}(\mu_k) \oplus \V_{-\Delta}(\mu_k)$ for $k\in\N$,
\item $\H^n=\bigoplus\limits_{k=1}^n\H_k$ for $n\in\N$.
\end{enumerate}

Fix $\lambda\in\Lambda=\H^G$. We will calculate the degree $\nabla_{N(K)}\text{-}\deg(\A(\cdot,\lambda), B(\im\pi_1))$, which is an element of the Euler ring $U(N(K))$. To do this we need to define an approximation scheme for the mapping $\A(\cdot,\lambda)$, see Appendix.
Consider the sequence of $N(K)$-equivariant orthogonal projections $\Gamma=\{\tau_n\colon\H\to\H:n\in\N_0\}$ defined as follows

\begin{enumerate}
\item $\H'^0=\{0\}$,
\item $\H'_k=\left(\V_{-\Delta}(\mu_k)^K \ominus\V_{-\Delta}(\mu_k)^G\right)\oplus \left(\V_{-\Delta}(\mu_k)^K\ominus\V_{-\Delta}(\mu_k)^G \right)$ for $k\in\N$,
\item $\H'^n=\bigoplus\limits_{k=1}^n \H'_k$ for $n\in\N$,
\item $\tau_n$ is a projection such that $\im\tau_n=\H'^n$, for $n\in\N$.
\end{enumerate}
Then $\Gamma$ is an  $N(K)$-equivariant approximation scheme on $\im\pi_1=\H^K\ominus\H^G$. Moreover,  $\ker L= \H^0$ and for every $n\in\N \cup \{0\}$ it follows that $\tau_n\circ L=L\circ \tau_n$. Note that $\pi_1(\H^n)=\H'^n$.

Consider the system:
\begin{equation}\label{lin}
\left\{ \begin{array}{rcl}
-\Delta w_1 =aw_1+bw_2& \text{in}&\Omega\\
\Delta w_2 = bw_1+cw_2& \text{in}&\Omega\\
\frac{\partial w_1}{\partial \nu}=\frac{\partial w_2}{\partial \nu}=0 & \text{on}& \partial \Omega
\end{array}\right.
\end{equation}
and put $A=
\left[\begin{array}{cc}
a &b\\
b& c
\end{array}\right]$.
Then
\[\Phi(w)= \frac{1}{2} \int\limits_{\Omega}|\nabla w_1(x)|^2 - |\nabla w_2(x)|^2 - \langle A(w(x)),w(x)\rangle dx.\]
Note that from \eqref{Phi} and \eqref{eta} it follows that $\nabla\Phi(w)=Lw-C_Aw$, where $C_A$ is given by
$\langle C_Aw,v\rangle_{\H}= \int\limits_{\Omega} \langle (L+A)w(x), v(x)\rangle  dx$ for $w,v \in \H$.

\begin{Lemma}
For every $w\in\H_k,\ v\in\H$, $\langle C_Aw, v\rangle_{\H}=\langle\frac{1}{1+\mu_k}(L+A)w,v\rangle_{\H}$.
\end{Lemma}

\begin{proof}
Note that $L+A=\left[\begin{array}{cc}
1+a &b\\
b& -1+c
\end{array}\right]$ and
consider the formula:
\[
\langle(L+A)w,v\rangle_{\H}=\int\limits_{\Omega}\langle\nabla (L+A)w(x), \nabla v(x) \rangle dx +\int\limits_{\Omega}\langle(L+A)w(x),v(x)\rangle dx.
\]
Then
\[
\left.\begin{array}{l}
\int\limits_{\Omega}\langle\nabla (L+A)w(x), \nabla v(x) \rangle dx \\
=\int\limits_{\Omega}\nabla ((1+a)w_1(x)) \nabla v_1(x) +\nabla (bw_1(x)) \nabla v_1(x)dx\\
+\int\limits_{\Omega}\nabla (bw_2(x)) \nabla v_2(x)+\nabla ((-1+c)w_2(x)) \nabla v_2(x) dx \\
=\int\limits_{\Omega}(-\Delta) ((1+a)w_1(x)) v_1(x) +(-\Delta)(bw_1(x)) v_1(x)dx\\
+\int\limits_{\Omega}(-\Delta) (bw_2(x))  v_2(x)+(-\Delta) ((-1+c)w_2(x))  v_2(x) dx
\\
=\mu_k\int\limits_{\Omega}((1+a)w_1(x)) v_1(x) +(bw_1(x)) v_1(x)+ (bw_2(x))  v_2(x)+ ((-1+c)w_2(x))  v_2(x) dx
\\
=\mu_k\int\limits_{\Omega}\langle(L+A)w(x),v(x)\rangle dx.
\end{array}\right.
\]
Therefore
\[
\langle(L+A)w,v\rangle_{\H}=(1+\mu_k)\int\limits_{\Omega}\langle(L+A)w(x)v(x)\rangle dx=(1+\mu_k) \langle C_Aw,v\rangle_{\H}.
\]
Hence
$\langle C_Aw, v\rangle_{\H}=\langle\frac{1}{1+\mu_k}(L+A)w,v\rangle_{\H}.$
\end{proof}

From the above lemma we obtain $ C_A (\H_k)\subset \H_k$ and therefore $ C_A \colon \H_k\to \H_k$.  To describe the restriction of $\A$ to subrepresentations of $\im\pi_1$, we first describe the restriction of $\nabla \Phi$ to subrepresentations of $\H$.  Let
$
T_k(A)=
\left[\begin{array}{cc}
1-\frac{a}{1+\mu_k} &-\frac{b}{1+\mu_k}\\-\frac{b}{1+\mu_k}& -1-\frac{c}{1+\mu_k}
\end{array}\right]$
and $\alpha_{1,k},\ \alpha_{2,k}$ be the eigenvalues of the matrix $T_k(A)$, $f_{1,k},\ f_{2,k}$ the corresponding eigenvectors. Because the matrix $T_k(A)$ is symmetric, $\alpha_{1,k},\ \alpha_{2,k}\in\R$.
Denote by $\epsilon_1,\ \epsilon_2$ the standard base of $\R^2$. Then $\H_k=\{\varphi_1(x)\cdot \epsilon_1+\varphi_2(x)\cdot \epsilon_2: \varphi_i\in\V_{-\Delta}(\mu_k)\}.$
It is easy to check that
\[\{\varphi_1(x)\cdot \epsilon_1+\varphi_2(x)\cdot \epsilon_2: \varphi_i\in\V_{-\Delta}(\mu_k)\}=\{\varphi_1(x)\cdot f_{1,k}+\varphi_2(x)\cdot f_{2,k}: \varphi_i\in\V_{-\Delta}(\mu_k)\}.\]
Hence we obtain
$
(\nabla\Phi)_{|\H_k}=
\left[\begin{array}{cc}
\alpha_{1,k} \Id &0\\
0& \alpha_{2,k}\Id
\end{array}\right],$
where $\Id\colon \V_{-\Delta}(\mu_{k}) \to \V_{-\Delta}(\mu_{k})$ is the identity map.

Now we are able to describe the action of the restrictions of $\A(\cdot,\lambda)$ on the subrepresentations of $\im\pi_1$.
Fix $\lambda\in\H^G$ and assume that $\dim \V_{-\Delta}(\mu_k)^K \ominus\V_{-\Delta}(\mu_k)^G>0$. Since $\A(u,\lambda)= \pi_1(\nabla\Phi(i(u,\lambda)))$,
\[\A_{|\H'_k}(u,\lambda)=\pi_1(\nabla\Phi_{|\H_k}(i((u_1,\lambda_1),(u_2,\lambda_2))))=\pi_1(\nabla\Phi_{|\H_k}((0,u_1,\lambda_1),(0,u_2,\lambda_2)))\]
\[=\pi_1(\alpha_{1,k} \Id(0,u_1,\lambda_1),  (\alpha_{2,k}\Id(0,u_2,\lambda_2)))=(\alpha_{1,k} u_1,  \alpha_{2,k}u_2),\]
where $(u,\lambda)= ((u_1,\lambda_1),(u_2,\lambda_2))$ and
$(u_i,\lambda_i)\in\left((\V_{-\Delta}(\mu_k)^K \ominus\V_{-\Delta}(\mu_k)^G)\oplus\V_{-\Delta}(\mu_k)^G\right)$ for $i=1,2$.
Therefore
$
(\A_{|\H'_k}(\cdot,\lambda))=
\left[\begin{array}{cc}
\alpha_{1,k} \Id &0\\
0& \alpha_{2,k}\Id
\end{array}\right],$
where $\Id\colon (\V_{-\Delta}(\mu_k)^K \ominus\V_{-\Delta}(\mu_k)^G) \to (\V_{-\Delta}(\mu_k)^K \ominus\V_{-\Delta}(\mu_k)^G)$.

Define $m^0(T_k(A))= \dim\ker T_k(A)$ and

\[m^0(\A_{|\H'_k}(\cdot,\lambda))=
\left\{\begin{array}{lcl}
\dim\ker\left(\left[\begin{array}{cc}
\alpha_{1,k} &0\\
0& \alpha_{2,k}
\end{array}\right]\right)
& \text{if} & \dim(\V_{-\Delta}(\mu_k)^K \ominus\V_{-\Delta}(\mu_k)^G)>0\\
0& \text{if} & \dim(\V_{-\Delta}(\mu_k)^K \ominus\V_{-\Delta}(\mu_k)^G)=0.
\end{array}\right.
\]
Put

$i^0(A)= \sum\limits_{k=1}^{\infty}m^0(T_k(A))$ and $\widetilde{i^0}(A)=
\sum\limits_{k=1}^{\infty}m^0(\A_{|\H'_k}(\cdot,\lambda))$.
It is easy to see that:
\begin{Lemma}
$\nabla\Phi$ is an isomorphism if and only if $i^0(A)= 0$. Fix  $\lambda\in\Lambda$. $\A(\cdot,\lambda)$ is an isomorphism if and only if $\widetilde{i^0}(A)= 0$.
\end{Lemma}
Naturally, if $\nabla\Phi$ is an isomorphism, so is $\A(\cdot,\lambda)$ for every $\lambda\in\Lambda$.

Denote by $m^-(T_k(A))$ the Morse index of the matrix $T_k(A)$.
 Note that $m^-(T_k(A))\in\{0,1,2\}$ and for a sufficiently large $k$, $m^-(T_k(A))=1$. Define the subspaces:
\[
\V_0(A)=\bigoplus\limits_{k\colon m^-(T_k(A))=0}
\V_{-\Delta}(\mu_k)^K\ominus\V_{-\Delta}(\mu_k)^G,\
 \ \V_2(A)=\bigoplus\limits_{k \colon m^-(T_k(A))=2}
 \V_{-\Delta}(\mu_k)^K\ominus\V_{-\Delta}(\mu_k)^G.\]

\begin{Theorem}\label{Theorem1}
 Consider system \eqref{lin} satisfying $\widetilde{i^0}(A)= 0$  and fix $\lambda\in\H^G$. Then
\[\nabla_{N(K)}\text{-}\deg(\A(\cdot,\lambda), B(\im\pi_1))=\nabla_{N(K)}\text{-}\deg(-\Id, B(V_2(A)))
\star\left(\nabla_{N(K)}\text{-}\deg(-\Id, B(V_0(A)))\right)^{-1}.\]
\end{Theorem}

\begin{proof}
From the definition of the degree, see formula \eqref{formulaofdegree}, for sufficiently large $n$ the following equality holds
\[
\left.\begin{array}{l}
\nabla_{N(K)}\text{-}\deg(\A(\cdot,\lambda), B(\im\pi_1))\\=
\left(\nabla_{N(K)}\text{-}\deg(L_2, B(\H'^n\ominus\H'^0))\right)^{-1}\star
\nabla_{N(K)}\text{-}\deg(\A_{|\H'^n}(\cdot,\lambda), B(\H'^n)).
\end{array}\right.
\]
Note that from the product formula, see Appendix, and from the definition of the function $L_2$ we obtain
\[
\nabla_{N(K)}\text{-}\deg(L_2, B(\H^n\ominus\H^0)) =
\nabla_{N(K)}\text{-}\deg(L_2, B(\bigoplus\limits^n_{k=1}
 \V_{-\Delta}(\mu_k)^K\ominus\V_{-\Delta}(\mu_k)^G))\] \[
=
\nabla_{N(K)}\text{-}\deg(-\Id,B(\V_{-\Delta}(\mu_1)^K\ominus\V_{-\Delta}(\mu_1)^G))\star \nabla_{N(K)}\text{-}\deg(-\Id,B(\V_{-\Delta}(\mu_2)^K\ominus\V_{-\Delta}(\mu_2)^G))\]
\[\star\ldots\star\nabla_{N(K)}\text{-}\deg(-\Id, B(\V_{-\Delta}(\mu_n)^K\ominus\V_{-\Delta}(\mu_n)^G)).
\]
If $m^-(T_k(A))=2$, then
\[
\left.\begin{array}{l}
\nabla_{N(K)}\text{-}\deg(\A_{|\H'_k}(\cdot,\lambda), B(\H'_k))\\
=\nabla_{N(K)}\text{-}\deg((-\Id,-\Id), B((\V_{-\Delta}(\mu_k)^K \ominus\V_{-\Delta}(\mu_k)^G)\oplus (\V_{-\Delta}(\mu_k)^K\ominus\V_{-\Delta}(\mu_k)^G)))\\
=
\nabla_{N(K)}\text{-}\deg(-\Id, B(\V_{-\Delta}(\mu_k)^K\ominus\V_{-\Delta}(\mu_k)^G))\star\nabla_{N(K)}\text{-}\deg(-\Id, B(\V_{-\Delta}(\mu_k)^K\ominus\V_{-\Delta}(\mu_k)^G))
\end{array}\right.
\]
When $m^-(T_k(A))=1$
\[
\nabla_{N(K)}\text{-}\deg(\A_{|\H'_k}(\cdot,\lambda), B(\H'_k))=
\nabla_{N(K)}\text{-}\deg(-\Id, B(\V_{-\Delta}(\mu_k)^K\ominus\V_{-\Delta}(\mu_k)^G)).
\]
In the remaining case $m^-(T_k(A))=0$ we have $\nabla_{N(K)}\text{-}\deg(\A_{|\H'_k}(\cdot,\lambda), B(\H'_k))=\I,$
which completes the proof.
\end{proof}

Consider the characteristic polynomial of $T_k(A)$ given by
\[W_k(\alpha_k)=(1-\frac{a}{1+\mu_k}-\alpha_k)(-1-\frac{c}{1+\mu_k}-\alpha_k)-\frac{b^2}{(1+\mu_k)^2}.\]
It is easy to verify that
\[
\left.\begin{array}{ll}
W_k(\alpha_k)&=\alpha_k^2 + \frac{a+c}{(1+\mu_k)}\alpha_k   + \frac{a-c}{(1+\mu_k)}  -1 +\frac{ac-b^2}{(1+\mu_k)^2} \\ & =\alpha_k^2 + \frac{a+c}{(1+\mu_k)}\alpha_k   +\frac{-(1+\mu_k)^2+(a-c)(1+\mu_k)+ac-b^2}{(1+\mu_k)^2}.
 \end{array}\right.\]
Since the matrix $T_k(A)$ is symmetric, the polynomial has two real roots, denote them by $\alpha_{1,k},\ \alpha_{2,k}$. From Viete's formulae we get
\[\alpha_{1,k}\alpha_{2,k}=\frac{-(1+\mu_k)^2+(a-c)(1+\mu_k)+ac-b^2}{(1+\mu_k)^2}, \ \ \alpha_{1,k}+\alpha_{2,k}=-\frac{a+c}{(1+\mu_k)}.\]
Because the sign of the sum does not depend on $k$, the matrix $T_k(A)$ has the roots of the same sign if and only if
$(1+\mu_k)^2-(a-c)(1+\mu_k)-(ac-b^2)<0$.
Solving the inequality (with respect to $1+\mu_k$) we obtain the discriminant
$\delta=(a-c)^2+4(ac-b^2)=(a+c)^2 - 4b^2$
and if $\delta\geq 0$, then $\beta_1=\frac{a-c-\sqrt{\delta}}{2}-1,\ \ \beta_2=\frac{a-c+\sqrt{\delta}}{2}-1$ are the roots of the polynomial $(1+\mu_k)^2-(a-c)(1+\mu_k)-(ac-b^2)$. 
If $\delta <0$, then we put $\beta_1=\beta_2=0$.
Note that if $(1+\mu_k)^2-(a-c)(1+\mu_k)-(ac-b^2) = 0$ for $\mu_k\in\sigma(-\delta,\Omega)$, then $\alpha_{1,k}=0$ or $\alpha_{2,k}=0$, so $i^0(A)\neq 0$ and if also $\dim(\V_{-\Delta}(\mu_k)^K \ominus\V_{-\Delta}(\mu_k)^G)>0$, then $\widetilde{i}^0(A)\neq 0$.

Let $P=\sigma(-\Delta,\Omega)\cap (\beta_1,\beta_2)$
and note that for every $\mu_k\in P$ we have
\begin{enumerate}
\item If $a+c<0$, then  $m^{-}(T_k(A))=2$.
\item If $a+c>0$, then  $m^{-}(T_k(A))=0$.
\end{enumerate}
Assume that $\widetilde{i^0}(A)=0$.

\begin{Theorem}
Under the above notations and assumptions:
\begin{enumerate}
\item if $a+c<0$, then
\[\nabla_{N(K)}\text{-}\deg(\A(\cdot,\lambda), B(\im\pi_1))=\nabla_{N(K)}\text{-}\deg(-\Id, B(\bigoplus\limits_{\mu\in P}\V_{-\Delta}(\mu)^K\ominus\V_{-\Delta}(\mu)^G)),\]
\item if $a+c>0$, then
\[\nabla_{N(K)}\text{-}\deg(\A(\cdot,\lambda), B(\im\pi_1))=\left(\nabla_{N(K)}\text{-}\deg(-\Id, B(\bigoplus\limits_{\mu\in P}\V_{-\Delta}(\mu)^K\ominus\V_{-\Delta}(\mu)^G))\right)^{-1},\]
\item
if the set $P$ is empty, then
$\nabla_{N(K)}\text{-}\deg(\A(\cdot,\lambda), B(\im\pi_1))=\I.$
\end{enumerate}

\end{Theorem}

Note that the condition $i^0(A)=0$ is satisfied if and only if for every $k\in \N$
\[-(1+\mu_k)^2+(a-c)(1+\mu_k)+ac-b^2\neq 0.
\]

From now on we consider the following nonlinear system 
\begin{equation}\label{system}
\left\{ \begin{array}{rcl}
-\Delta w_1 = \nabla_{w_1} F(w_1,w_2)& \text{in}&\Omega\\
\Delta w_2=\nabla_{w_2} F(w_1,w_2)& \text{in}&\Omega\\
\frac{\partial w_1}{\partial \nu}=\frac{\partial w_2}{\partial \nu}=0 & \text{on}& \partial \Omega,
\end{array}\right.
\end{equation}
Recall that with this system is associated the functional $\Phi \colon \H\to \R$ given by \[\Phi(w)=  \int\limits_{\Omega}\frac{1}{2}(|\nabla w_1(x)|^2 - |\nabla w_2(x)|^2) - F(w(x))dx.\]
and $\nabla \Phi (w) = L w -\nabla\eta(w)$, where $\nabla \eta$ is an operator defined by
\[
\langle \nabla\eta(w),v \rangle_{\H}= \int\limits_{\Omega}\langle L w(x),v(x)\rangle + \langle\nabla F(w(x)),v(x)\rangle dx.
\]
Moreover, $\nabla^2 \Phi (w) = L  -C_{\nabla^2 F(w)}$ for $w\in\H$, where the operator $C_{\nabla^2 F(w)}\colon\H\to\H$ is given by the equality
\[\langle C_{\nabla^2 F(w)}u,v\rangle_{\H}= \int\limits_{\Omega} \langle (L+\nabla^2 F(w(x)))u(x), v(x)\rangle  dx \text{ for } u,v \in \H. \]
Denote $\z=(\nabla F)^{-1}(0)$ and let $z\in \z$ be a non-degenerate critical point of the functional $\Phi$. Define
\[\langle C_{\nabla^2 F(z)}u,v\rangle_{\H}= \int\limits_{\Omega} \langle (L+\nabla^2 F(z))u(x), v(x)\rangle  dx \text{ for } u,v \in \H. \]
Denote
$\nabla^2 F(z)=
\left[\begin{array}{cc}
a(z) &b(z)\\
b(z)& c(z)
\end{array}\right]$.
Then
$T_k(\nabla^2 F(z))=
\left[\begin{array}{cc}
1-\frac{a(z)}{1+\mu_k} &-\frac{b(z)}{1+\mu_k}\\-\frac{b(z)}{1+\mu_k}& -1-\frac{c(z)}{1+\mu_k}
\end{array}\right]$
and let $\alpha_{1,k}(z),$ $\alpha_{2,k}(z)$ be the (real) eigenvalues of the matrix $T_k(\nabla^2 F(z))$.
Then
\[
(\nabla^2\Phi(z))_{|H_k}=
\left[\begin{array}{cc}
\alpha_{1,k}(z) \Id &0\\
0& \alpha_{2,k}(z)\Id
\end{array}\right],\]
where $\Id\colon \V_{-\Delta}(\mu_{k}) \to \V_{-\Delta}(\mu_{k})$ is the identical function.
Note that for a sufficiently large $k$, $m^-(T_k(\nabla^2 F(z)))=1$.

Define $\lambda\in\H^G$ by $\lambda(x)=z$ for every $x\in\Omega$.
Since the derivative of $\A$ with respect to $u$ satisfies $\A_u'(0,\lambda)=\pi_1\circ\nabla^2\Phi(0,0,\lambda)\circ i$, it follows that if $\dim(\V_{-\Delta}(\mu_k)^K \ominus\V_{-\Delta}(\mu_k)^G)>0$, then
\[
(\A_{|\H'_k}(\cdot,\lambda))=
\left[\begin{array}{cc}
\alpha_{1,k}(z) \Id &0\\
0& \alpha_{2,k}(z)\Id
\end{array}\right],\]
where $\Id\colon (\V_{-\Delta}(\mu_k)^K \ominus\V_{-\Delta}(\mu_k)^G) \to (\V_{-\Delta}(\mu_k)^K \ominus\V_{-\Delta}(\mu_k)^G)$.

Define the subspaces
\[\V_0(\nabla^2 F(z))=\bigoplus\limits_{k\colon m^-(T_k(\nabla^2 F(z)))=0}\V_{-\Delta}(\mu_k)^K\ominus\V_{-\Delta}(\mu_k)^G,\] \[ \V_2(\nabla^2 F(z))=\bigoplus\limits_{k\colon m^-(T_k(\nabla^2 F(z)))=2}\V_{-\Delta}(\mu_k)^K\ominus\V_{-\Delta}(\mu_k)^G.\]

\begin{Theorem}
Let $z\in \z$ be such that $\widetilde{i^0}(\nabla^2 F(z))= 0$. Then there exists $\gamma_0$ such that for every $0<\gamma<\gamma_0$
\[
\left.\begin{array}{l}\nabla_{N(K)}\text{-}\deg(\A(\cdot,\lambda), B_{\gamma}(\im\pi_1))=\nabla_{N(K)}\text{-}\deg(\A_u'(0,\lambda), B(\im\pi_1))\\
=\nabla_{N(K)}\text{-}\deg(-\Id, B(V_2(\nabla^2 F(z)))\star\left(\nabla_{N(K)}\text{-}\deg(-\Id, B(V_0(\nabla^2 F(z)))\right)^{-1}.
\end{array}\right.
\]
\end{Theorem}
\begin{proof}
We refer the reader to \cite{degree} for the proof of the first equality. The second equality follows from theorem \ref{Theorem1}.
\end{proof}

Put $\delta(z)=(a(z)-c(z))^2+4(a(z)c(z)-b(z)^2)=(a(z)+c(z))^2-4b(z)^2$ and
if $\delta(z)>0$, then put $\beta_1(z)=\frac{a(z)-c(z)-\sqrt{\delta(z)}}{2}-1,\ \ \beta_2(z)=\frac{a(z)-c(z)+\sqrt{\delta(z)}}{2}-1$.
If $\delta(z) <0$, then put $\beta_1(z)=\beta_2(z)=0$. Define $P(z)=\sigma(-\Delta,\Omega)\cap (\beta_1(z),\beta_2(z))$. Similarly as before it can be shown that the matrix $T_k(\nabla^2 F(z))$ has roots of the same sign if and only if $\mu_k\in P(z)$. Moreover, for any $\mu_k\in P(z)$:
\begin{enumerate}
\item If $a(z)+c(z)<0$, then  $m^{-}(T_k(\nabla^2 F(z)))=2$.
\item If $a(z)+c(z)>0$, then  $m^{-}(T_k(\nabla^2 F(z)))=0$.
\end{enumerate}

Now we are able to give formulae of the degree of the functional associated with system \eqref{system}. Assume that $\widetilde{i^0}(\nabla^2 F(z))=0$.

\begin{Theorem}\label{Theorem2}
Under the above notions and assumptions:
\begin{enumerate}
\item if $a(z)+c(z)<0$, then for a sufficiently small $\gamma>0$
\[\nabla_{N(K)}\text{-}\deg(\A(\cdot,\lambda), B_{\gamma}(\im\pi_1))=
\nabla_{N(K)}\text{-}\deg(-\Id, B(\bigoplus\limits_{\mu\in P(z)}\V_{-\Delta}(\mu_k)^K\ominus\V_{-\Delta}(\mu_k)^G)),\]
\item if $a(z)+c(z)>0$, then for a sufficiently small $\gamma>0$
\[\nabla_{N(K)}\text{-}\deg(\A(\cdot,\lambda), B_{\gamma}(\im\pi_1))=
\left(\nabla_{N(K)}\text{-}\deg(-\Id, B(\bigoplus\limits_{\mu\in P(z)}\V_{-\Delta}(\mu_k)^K\ominus\V_{-\Delta}(\mu_k)^G))\right)^{-1},\]
\item
if the set $P(z)$ is empty, then for a sufficiently small $\gamma>0$
\[\nabla_{N(K)}\text{-}\deg(\A(\cdot,\lambda), B_{\gamma}(\im\pi_1))=\I.\]
\end{enumerate}
\end{Theorem}

The following theorem will be used to formulate a bifurcation theorem. The proof is similar in spirit to that of \cite{Garza}. We use the notation from this article.
\begin{Lemma}\label{connected}
Let the above assumptions hold and assume that the group $G$ is connected.  If $\W_1$ or $\W_2$ is a nontrivial $G$-representation, then
\[\nabla_{G}\text{-}\deg(-\Id, B(\W_1)) \neq \nabla_{G}\text{-}\deg(-\Id, B(\W_2))^{-1}.\]
\end{Lemma}
\begin{proof}
Recall that $\nabla_{G}\text{-}\deg(-\Id, B(\W_i))=\chi_{G}(S^{\W_i})$ for $i=1,2$, see \cite{Geba}, where $S^{\W_i} = D(\W_i)/S(\W_i)$ is a quotient of the closed unit ball and the unit sphere in $\W_i$ and $\chi_{G}$ denotes the $G$-invariant Euler characteristic of $S^{\W_i}$. Suppose the assertion of the lemma is false, that is $\chi_T(S^{\W_1})\star\chi_T(S^{\W_2})=\mathbb{I}\in U(G)$, the unit in $U(G)$. Denote by $T\subset G$ a maximal torus and notice that because $\W_i$ are $G$-representations, $\W_i$ can be also treated as $T$-representations. The natural homomorphism $\varphi\colon T \to G$ induces a ring homomorphism $\varphi^*\colon U(T) \to U(G)$ such that $\varphi^*(\chi_{G}(S^{\W_i}))=\chi_T(S^{\W_i})$. Hence $\chi_T(S^{\W_1})\star\chi_T(S^{\W_2})=\mathbb{I}\in U(T)$. From theorem 3.2 of \cite{Garza} it follows that there exist $k_0,\ k_1, \ldots, k_r,\ k'_0,\ k'_1, \ldots, k'_{r'}\in \N\cup\{0\}$, $K_{m_1},\ldots,K_{m_r},\ K_{m'_1},\ldots,K_{m'_{r'}}\in\sub(T)$, $x,\ x',\ y,\ y'\in U(T)$ such that
 \begin{enumerate}
 \item $\dim K_{m_i}=\dim K_{m'_i}=\dim T - 1$ for every $i$,
 \item $x=\sum\limits^r_{i=1} k_i \cdot \chi_T(T/K_{m_i})$,
 \item $x'=\sum\limits^{r'}_{i=1} k'_i \cdot \chi_T(T/K_{m'_i})$,
 \item $y=\sum\limits_{(K)\in\{(\mathcal{K})\in\sub[T]:\dim \mathcal{K} <\dim T -1 \}} n(K)\cdot\chi_T(T/K^+)$, where $n(K)\in\Z$,
 \item $y'=\sum\limits_{(K)\in\{(\mathcal{K})\in\sub[T]:\dim \mathcal{K} <\dim T -1 \}} n'(K)\cdot\chi_T(T/K^+)$, where $n'(K)\in\Z$,
\item $\chi_T(S^{\W_1})=(-1)^{k_0}\mathbb{I}+(-1)^{k_0}x +y$,
\item $\chi_T(S^{\W_2})=(-1)^{k'_0}\mathbb{I}+(-1)^{k'_0}x' +y'$.
 \end{enumerate}
 From the assumptions (since $G$ is connected) it follows that $x\neq\mathbb{I}$ or $x'\neq\mathbb{I}$.
  It is easy to see that \[\left.\begin{array}{l}
  ((-1)^{k_0}\mathbb{I}+(-1)^{k_0}x+y)((-1)^{k'_0}\mathbb{I}+(-1)^{k'_0}x'+y')\\
=(-1)^{k_0+k'_0}\mathbb{I}+(-1)^{k_0+k'_0}x+(-1)^{k_0+k'_0}x'\\
+(-1)^{k'_0}y+(-1)^{k_0}y'+(-1)^{k_0+k'_0}xx'+(-1)^{k_0}xy'+(-1)^{k'_0}x'y+yy'.\end{array}\right.\]
  Note that if $k_0+k'_0$ is even, then $(-1)^{k_0+k'_0}=1$ and 
 \[\left.\begin{array}{l}
  (-1)^{k_0+k'_0}\mathbb{I}+(-1)^{k_0+k'_0}x+(-1)^{k_0+k'_0}x'= \mathbb{I}+x+x'\neq \mathbb{I}.
  \end{array}\right.\]
  If $k_0+k'_0$ is odd, then $(-1)^{k_0+k'_0}=-1$ and
   \[\left.\begin{array}{l}
  (-1)^{k_0+k'_0}\mathbb{I}+(-1)^{k_0+k'_0}x+(-1)^{k_0+k'_0}x'= -\mathbb{I}-x-x'\neq \mathbb{I}.
  \end{array}\right.
  \]
  From the properties of the multiplication in $U(T)$, see \cite{Garza}, it is easy to verify that \[(-1)^{k'_0}y+(-1)^{k_0}y'+xx'+xy'+x'y+yy'=\sum\limits_{(K)\in\{(\mathcal{K})\in\sub[T]:\dim \mathcal{K} <\dim T -1 \}} n''(K)\cdot\chi_T(T/K^+),\] where $n''(K)\in\Z$. Therefore $\chi_T(S^{\W_1})\star\chi_T(S^{\W_2})=((-1)^{k_0}\mathbb{I}+x+y)((-1)^{k'_0}\mathbb{I}+x'+y') \neq \mathbb{I}$. This contradiction completes the proof.
\end{proof}

\section{The main results}
In this section we apply the bifurcation theory to formulate conditions implying breaking of symmetries. That, is we formulate theorems to answer the problem: does there exist a function $f=(f_1,f_2)\in\H^{G}$ such that the system   
\begin{equation}\label{rowNeuSN1}
\left\{ \begin{array}{rcl}
-\Delta w_1 = \nabla_{w_1} F(w_1,w_2)+f_1& \text{in}&\Omega\\
\Delta w_2=\nabla_{w_2} F(w_1,w_2)+f_2& \text{in}&\Omega\\
\frac{\partial w_1}{\partial \nu}=\frac{\partial w_2}{\partial \nu}=0 & \text{on}& \partial \Omega,
\end{array}\right.
\end{equation}
has a weak solution  $w\in \H^K \backslash \H^{G}$?

Recall that $\Phi\colon\H^1(\Omega)\times\H^1(\Omega)\to\R$ is the functional associated with the above system and
\begin{enumerate}
\item $\Omega$ is an open, bounded and $G$-invariant subset of an orthogonal $G$-representation $\R^n$, with a smooth boundary,
\item $F\in C^2(\R^2,\R)$,
\item  $|\nabla^2 F(y)|\leq a +b|y|^q$, where $a,\ b\in\R,\ q<\frac{4}{n-2}$ for $n\geq 3$ and $q<\infty$ for $n=2$.
\end{enumerate}

Assume that the set $\z=(\nabla F)^{-1}(0)$ is finite and fix $z\in\z$.
In the previous sections we have introduced the following notation: $\nabla^2 F(z)=
\left[\begin{array}{cc}
a(z) &b(z)\\
b(z)& c(z)
\end{array}\right]$,
$\delta(z)=(a(z)+c(z))^2-4b(z)^2$ and
 if $\delta(z)>0$, then $\beta_1(z)=\frac{a(z)-c(z)-\sqrt{\delta(z)}}{2}-1,\ \ \beta_2(z)=\frac{a(z)-c(z)+\sqrt{\delta(z)}}{2}-1$.
 If $\delta(z) <0$, then $\beta_1(z)=\beta_2(z)=0$. We have also put $P(z)=\sigma(-\Delta,\Omega)\cap (\beta_1(z),\beta_2(z))$.

Denote
\[
m_1=\min\{\beta_1(z):z\in\z\},\ \ M_1=\max\{\beta_1(z):z\in\z\},
\]
\[
m_2=\min\{\beta_2(z):z\in\z\},\ \ M_2=\max\{\beta_2(z):z\in\z\}.
\]

Assume that $P(z)\neq\emptyset$ and $\gamma>0$ is a sufficiently small number.
By theorem \ref{Theorem2} we obtain
\begin{equation}\label{degree}
\left.\begin{array}{l}\nabla_{N(K)}\text{-}\deg(\A(\cdot,\lambda), B_{\gamma}(\im\pi_1))=\\
 \\
=\left\{
\begin{array}{l}
\nabla_{N(K)}\text{-}\deg(-\Id, B(\bigoplus\limits_{\mu\in P(z)}(\V_{-\Delta}(\mu)^K\ominus\V_{-\Delta}(\mu)^G)))\ \text{, when } a(z)+c(z)<0\\
\left(\nabla_{N(K)}\text{-}\deg(-\Id, B(\bigoplus\limits_{\mu\in P(z)}(\V_{-\Delta}(\mu)^K\ominus\V_{-\Delta}(\mu)^G)))\right)^{-1}\ \text{, when } a(z)+c(z)>0.
\end{array}\right.
\end{array}\right.
\end{equation}
If $P(z)= \emptyset$, then $\nabla_{N(K)}\text{-}\deg(\A(\cdot,\lambda), B_{\gamma}(\im\pi_1))=\mathbb{I}\in U(N(K)).$

Recall that for every $\lambda\in\Lambda$, $\A(0,\lambda)=0$.
In the following theorems we formulate conditions implying  a global bifurcation of solutions of the equation $\A(u,\lambda)=0$.
To do this we use coefficients of the matrix $\nabla^2 F(z)$. We use the notation: if $z_1,\ z_2\in \z$, then $\lambda_1,\ \lambda_2\in\H^G$ are defined by $\lambda_1(x)=z_1$  and $\lambda_2 (x)=z_2$ for every $x\in\Omega$.
\begin{Theorem}\label{lam1}
Let the above assumptions hold. Assume that there exist $z_1,\ z_2\in \z$ such that
 \begin{enumerate}
  \item $P(z_1)\neq\emptyset$ and $\bigoplus\limits_{\mu\in P(z_1)}\V_{-\Delta}(\mu)$ is a nontrivial $N(K)$-representation or $P(z_2)\neq\emptyset$ and $\bigoplus\limits_{\mu\in P(z_2)}\V_{-\Delta}(\mu)$ is a nontrivial $N(K)$-representation,
  \item $a(z_1)+c(z_1)<0<a(z_2)+c(z_2)$,
 \item $\widetilde{i^0}(\nabla^2 F(z_1))=\widetilde{i^0}(\nabla^2 F(z_2))=0$,
 \item $N(K)$ is a connect group.
 \end{enumerate}
 Then there exists a global bifurcation point of solutions of the equation $\A(u,\lambda)=0.$
 \end{Theorem}

\begin{proof}
From formula \eqref{degree} it follows that for sufficiently small $\gamma_1,\ \gamma_2>0$ :
\[
\nabla_{N(K)}\text{-}\deg(\A(\cdot,\lambda_1), B_{\gamma_1}(\im\pi_1))=\nabla_{N(K)}\text{-}\deg(-\Id, B(\bigoplus\limits_{\mu\in P(z_1)}(\V_{-\Delta}(\mu)^K\ominus\V_{-\Delta}(\mu)^G)))
\]
and
\[
\nabla_{N(K)}\text{-}\deg(\A(\cdot,\lambda_2), B_{\gamma_2}(\im\pi_1))=\nabla_{N(K)}\text{-}\deg(-\Id, B(\bigoplus\limits_{\mu\in P(z_2)}(\V_{-\Delta}(\mu)^K\ominus\V_{-\Delta}(\mu)^G)))^{-1}
\]
From lemma \ref{connected} we obtain that \[\nabla_{N(K)}\text{-}\deg(\A(\cdot,\lambda_1), B_{\gamma_1}(\im\pi_1))\neq\nabla_{N(K)}\text{-}\deg(\A(\cdot,\lambda_2),B_{\gamma_2}(\im\pi_1)).\] Therefore the theorem follows from theorem \ref{GLOB}.
\end{proof}

Also, if between $m_1$ and $M_1$ exists $\mu\in\sigma(-\Delta,\Omega)$, then there can appear breaking of symmetries.

\begin{Theorem}\label{lam2}
Suppose that for every $z\in\z$ either $a(z)+c(z)>0$ or $a(z)+c(z)<0$.
Let $z_1,\ z_2\in \z$ be such that $m_1=\beta_1(z_1),\ M_1=\beta_1(z_2)$ and assume that $\widetilde{i^0}(\nabla^2 F(z_1))=\widetilde{i^0}(\nabla^2 F(z_2))=0.$ Suppose that one of the following conditions is satisfied:
\begin{enumerate}
\item There exists $\mu_i\in\sigma(-\Delta,\Omega)$ such that $m_1<\mu_i<M_1$, where
$\V_{-\Delta}(\mu_i)^K\ominus\V_{-\Delta}(\mu_i)^G$ is a nontrivial $N(K)$-representation and for every $\mu_j\in\sigma(-\Delta,\Omega)$
if
$m_2<\mu_j<M_2$, then $\V_{-\Delta}(\mu_j)^K\ominus\V_{-\Delta}(\mu_j)^G$ is a trivial, even-dimensional $N(K)$-representation.

\item There exist $\mu_i,\ \mu_j\in\sigma(-\Delta,\Omega)$ such that $m_1<\mu_i<M_1$ and $\mu_j<m_2\leq M_2<\mu_{j+1}$, where
$\V_{-\Delta}(\mu_i)^K\ominus\V_{-\Delta}(\mu_i)^G$ is a nontrivial $N(K)$-representation.

\end{enumerate}
 Then there exists a global bifurcation point of solutions of the equation $\A(u,\lambda)=0.$
\end{Theorem}

\begin{proof}
It easy to see that both the assumptions imply that for sufficiently small $\gamma_1,\ \gamma_2>0$  \[\nabla_{N(K)}\text{-}\deg(\A(\cdot,\lambda_1), B_{\gamma_1}(\im\pi_1))\neq\nabla_{N(K)}\text{-}\deg(\A(\cdot,\lambda_2),B_{\gamma_2}(\im\pi_1)).\] Therefore the thesis follows from theorem \ref{GLOB}.
\end{proof}

Analogously as the previous theorem, if between $m_2$ and $M_2$ exists $\mu\in\sigma(-\Delta,\Omega)$ we can formulate and prove conditions implying breaking of symmetries.

\begin{Theorem}\label{lam3}
Suppose that for every $z\in\z$ either $a(z)+c(z)>0$ or $a(z)+c(z)<0$.
Let $z_1,\ z_2\in \z$ be such that $m_2=\beta_2(z_1),\ M_2=\beta_2(z_2)$ and assume that $\widetilde{i^0}(\nabla^2 F(z_1))=\widetilde{i^0}(\nabla^2 F(z_2))=0.$ Suppose that one of the following conditions is satisfied:
\begin{enumerate}
\item There exists $\mu_j\in\sigma(-\Delta,\Omega)$ such that $m_2<\mu_j<M_2$, where
$\V_{-\Delta}(\mu_j)^K\ominus\V_{-\Delta}(\mu_j)^G$ is a nontrivial $N(K)$-representation and for every $\mu_i\in\sigma(-\Delta,\Omega)$ if
$m_1<\mu_i<M_1$, then $\V_{-\Delta}(\mu_i)^K\ominus\V_{-\Delta}(\mu_i)^G$ is a trivial, even-dimensional $N(K)$-representation.
\item There exist $\mu_i,\ \mu_j\in\sigma(-\Delta,\Omega)$ such that $\mu_i<m_1\leq M_1<\mu_{i+1}$ and $m_2<\mu_j<M_2$, where
$\V_{-\Delta}(\mu_j)^K\ominus\V_{-\Delta}(\mu_j)^G$ is a nontrivial $N(K)$-representation.

\end{enumerate}
 Then there exists a global bifurcation point of solutions of the equation $\A(u,\lambda)=0.$
\end{Theorem}

Recall that if there exists a global bifurcation point of solutions of the equation $\A(u,\lambda)=0$, then  there exists $\lambda_0\in\Lambda$ such that the connected component $C(\lambda_0)$ in the closure of the set $\mathcal{N}=\{(u,\lambda)\in (\im\pi_1\setminus\{0\})\times\Lambda: \A(u,\lambda)=0\}$ satisfies $C(\lambda_0) \neq \{(0, \lambda_0)\}$ and either   $C(\lambda_0)$ is not bounded or $C(\lambda_0) \cap ({0}\times(\Lambda \setminus \{\lambda_0\}) \neq \emptyset$. In particular, we obtain a connected set of solutions of our main problem.

Moreover, for every $(u,\lambda)\in C(\lambda_0)\setminus(\{0\}\times\Lambda)$, the equality $\A(u,\lambda)=0$ implies that
$\pi_1(\nabla\Phi(i(u,\lambda)))=0$, so $\nabla\Phi(i(u,\lambda))\in \H^G$ and therefore there exists $f=(f_1,f_2)\in\H^G$ such that the system
\begin{equation}\label{111}
\left\{ \begin{array}{rcl}
-\Delta w_1 = \nabla_{w_1} F(w_1,w_2)+f_1 & \text{in}&\Omega\\
\Delta w_2=\nabla_{w_2} F(w_1,w_2)+f_2 & \text{in}&\Omega\\
\frac{\partial w_1}{\partial \nu}=\frac{\partial w_2}{\partial \nu}=0 & \text{on}& \partial \Omega,
\end{array}\right.
\end{equation}
has a solution in  $\H^K\backslash \H^G$. That is, $i(C(\lambda_0)\setminus(\{0\}\times\Lambda))$ is a set of solutions of the main problem. Moreover, we have obtained that the whole connected set $i(C(\lambda_0))$ is mapped by $\nabla\Phi$ into a connected set in $\H^G$.

\section{Examples}
In this section we show an application of the bifurcation theorems.  Consider the following system
\begin{equation}\label{rowNeukula}
\left\{ \begin{array}{rcl}
-\Delta w_1 = \nabla_{w_1} F(w_1,w_2)& \text{in}&B^n\\
\Delta w_2=\nabla_{w_2} F(w_1,w_2)& \text{in}&B^n\\
\frac{\partial w_1}{\partial \nu}=\frac{\partial w_2}{\partial \nu}=0 & \text{on}& \partial S^{n-1},
\end{array}\right.
\end{equation}
where
\begin{enumerate}
\item $B^n$ is an open unit ball in an orthogonal $\SO(n)$-representation $\R^n$, where $\SO(n)$ is a special orthogonal group,
\item $F\in C^2(\R^2,\R)$,
\item $|\nabla^2 F(y)|\leq a +b|y|^q$, where $a,b\in\R,\ q<\frac{4}{n-2}$ for $n\geq 3$ and $q<\infty$ for $n=2$.
\end{enumerate}
Denote $\z=(\nabla F)^{-1}(0)$ and define the function $\varphi\colon\R\times\z\to\R$ by the formula $\varphi(x,z)=-x^2+(a(z)-c(z))x+a(z)c(z)-b(z)^2$. Put $\K=\{1+\mu:\mu\in\sigma(-\Delta,B^n)\}$. Recall that the set $\sigma(-\Delta,B^n)$ is discrete and therefore so is the set $\K$.

\begin{Theorem}\label{Example}
Suppose that there exist $z_1,\ z_2\in \z$ such that
\begin{enumerate}
\item $\varphi(\cdot, z_i)^{-1}(0)\cap \K=\emptyset$ for $i=1, 2$,
\item $P(z_1)\neq\emptyset$ and $\bigoplus\limits_{\mu\in P(z_1)}\V_{-\Delta}(\mu)$ is a nontrivial $N(K)$-representation or $P(z_2)\neq\emptyset$ and $\bigoplus\limits_{\mu\in P(z_2)}\V_{-\Delta}(\mu)$ is a nontrivial $N(K)$-representation,
\item $a(z_1)+c(z_1)<0<a(z_2)+c(z_2)$.
\end{enumerate}
Then there exists a connected set of points in $\H^K\backslash \H^G$ such that for each point from this set there exists $f\in\H^G$ such that these points are solutions of the system
\begin{equation}\label{1}
\left\{ \begin{array}{rcl}
-\Delta w_1 = \nabla_{w_1} F(w_1,w_2) +f_1& \text{in}&\Omega\\
\Delta w_2=\nabla_{w_2} F(w_1,w_2) +f_2& \text{in}&\Omega\\
\frac{\partial w_1}{\partial \nu}=\frac{\partial w_2}{\partial \nu}=0 & \text{on} &\partial \Omega.
\end{array}\right.
\end{equation}
\end{Theorem}

\begin{proof}
The first assumption implies that $i^0(\nabla^2 F(z_1))=i^0(\nabla^2 F(z_2))=0$, hence $\widetilde{i^0}(\nabla^2 F(z_1))=\widetilde{i^0}(\nabla^2 F(z_2))=0$. Theorem \ref{lam1} completes the proof.
\end{proof}

Elements of the set $\sigma(-\Delta, B^n)$ are well known, especially if $n=2,3$, see \cite{Michlin}. So are representations of $\SO(n)$, for $n=2,3$, and the degree for $\SO(2)$-invariant functionals. Therefore the assumptions of theorem \ref{Example} are easy to verify.

\begin{Remark}
\begin{enumerate}

\item If $G=\SO(n)$ and $K=\{e\}$ (the trivial subgroup), then the normalizer of $K$ is equal to $\SO(n)$ and $\H^K=\H$. Therefore using our method we can study existing of non-radial solutions such that $\nabla\Phi(w)\in \H^{\SO(n)}$.
\item If $G=\SO(3)$ and $K=\Z_m$ (a cyclic group), then the Weyl group of $K$ is equal to $\operatorname{O}(2)$ and the space $\H^K$ is an $\operatorname{O}(2)$-representation and therefore an $\SO(2)$-representation. Using the degree for $\SO(2)$-equivariant maps we can study the problem: does there exist $w\in \H^{\Z_m}\backslash \H^{\SO(3)}$ such that $\nabla\Phi(w)\in \H^{\SO(3)}$?
    \item In \cite{Dancer} has been considered problem of breaking symmetries for elliptic equation using the homotopy index. The author has obtained a sequence of solution of the bifurcation problem. We emphasize that using the degree for gradient $G$-equivariant maps, we have obtained a connected set of solutions.
\item In this article we have considered a simple example of the method, but the same could be used  for more complicated equations.
\end{enumerate}
\end{Remark}

\section{Appendix}
To make this article self-contained we recall the definitions and properties of the Euler ring for a compact Lie group and the degree for $G$-invariant strongly indefinite functionals.

Denote by $\F(G)$ the class of pointed $G$-CW-complexes, see \cite{Dieck} for the definition, and by $[X]$ the $G$-homotopy class of a pointed $G$-CW-complex $X$. Let $\mathbf{F}$ be the free abelian group generated by the pointed $G$-homotopy types of finite $G$-CW-complexes and $\mathbf{N}$ the subgroup of $\mathbf{F}$ generated by all elements $[A]-[X]+[X/A]$ for pointed $G$-CW-subcomplexes $A$ of a pointed $G$-CW-complex $X$.

\begin{Definition}
Put $U(G)=\mathbf{F}/\mathbf{N}$ and let $\chi_G(X)$ be the class of $[X]$ in $U(G)$. The element $\chi_G(X)$ is said to be the $G$-equivariant Euler characteristic of a pointed $G$-CW-complex $X$.

For $X,Y\in\F(G)$ let $[X\vee Y]$ denote a $G$-homotopy type of the wedge $X\vee Y\in\F(G)$. Since $[X]-[X\vee Y]+[(X\vee Y)/X]=[X]-[X\vee Y]+[Y] \in \mathbf{N}$
\begin{equation}\label{plus}
\chi_G(X)+ \chi_G(Y)=\chi_G(X\vee Y).
\end{equation}

For $X,Y\in\F(G)$ let $X\wedge Y=X\times/X\vee Y$, The assignment $(X,Y)\mapsto X\wedge Y$ induces a product $U(G)\times U(G)\to U(G)$ given by
\begin{equation}\label{razy}
\chi_G(X)\star \chi_G(Y)=\chi_G(X\wedge Y).
\end{equation}
\end{Definition}

If $X$ is a $G$-CW-complex without a base point, then by $X^+$ we denote a pointed $G$-CW-complex $X\cup\{\star\}$ and consequently we put $\chi_G(X)=\chi_G(X^+)$.
\begin{Lemma}
$(U(G),+,\star)$ with an additive and multiplicative structures given by \eqref{plus}, \eqref{razy}, respectively, is a commutative ring with unit $\mathbb{I}=\chi_G(G/G^+)$.
\end{Lemma}

We call $(U(G),+,\star)$ the Euler ring of $G$.

Denote by $\sub[G]$ the set of conjugacy classes of subgroups of a group $G$.

\begin{Lemma}
$(U(G),+)$ is a free abelian group with basis $\chi_G(G/K^+),\ (K)\in\sub[G]$.
\end{Lemma}

See \cite{Dieck1, Dieck} for the complete definition and more properties of the Euler ring.

An element of the Euler ring is the degree for  $G$-invariant strongly indefinite functionals, we recall the definition and properties.
Let $(\H,\langle\cdot,\cdot\rangle)$ be an infinite-dimensional, separable Hilbert space which is an orthogonal $G$-representation. Denote by $\Gamma=\{\tau_n\colon\H\to\H:n\in\N \cup \{0\}\}$  a sequence of $G$-equivariant orthogonal projections.
\begin{Definition}
A set $\Gamma$ is said to be a $G$-equivariant approximation scheme on $\H$ if
\begin{enumerate}
\item for every $n\in\N \cup \{0\}$, $\H^n$ is a finite subrepresentation of the representation $\H$,
\item $\H^{n+1}=\H^n\oplus\H_{n+1}$ and $\H^n\bot\H_{n+1}$,
\item for every $u\in\H \lim\limits_{n\rightarrow\infty}\tau_n(u)=u$.\\
\end{enumerate}
\end{Definition}
Assume that
\begin{description}
\item[(a1)] $\Omega\subset\H$ is an open, bounded and $G$-invariant subset,
\item[(a2)] $L\colon\H\to\H$ is a linear, bounded, self-adjoint, $G$-equivariant Fredholm operator satisfying the following assumptions:
\begin{enumerate}[(a)]
\item $\ker L=\H^0$,
\item $\pi_n\circ L=L\circ\pi_n$, for all $n\in\N \cup \{0\}$,
\end{enumerate}
\item[(a3)] $\nabla\eta\colon\Omega\to \H$ is a continuous, $G$-equivariant, compact operator,
\item[(a4)] $\Phi\in C^1_G(\Omega,\R)$ satisfies the following assumptions:
\begin{enumerate}[(a)]
\item $\nabla\Phi(u)=Lu-\nabla\eta(u)$,
\item $\operatorname{cl}((\nabla\Phi)^{-1}(0))\cap\partial\Omega=\emptyset$.
\end{enumerate}
\end{description}

Under the above assumptions define the degree for $G$-invariant strongly indefinite functionals by
\begin{equation}\label{formulaofdegree}
\nabla_G\text{-}\deg(L-\nabla \eta, \Omega)=
(\nabla_G\text{-}\deg(L, B(\H^n\ominus\H^0)))^{-1}\star
\nabla_G\text{-}\deg(L-\pi_n\nabla\eta, \Omega_{\epsilon}\cap\H^n),
\end{equation}
where $\epsilon>0$ is sufficiently small and $n\in \N$ is sufficiently large, see \cite{degree} for details.

\begin{Theorem}
The degree has the following properties:
\begin{enumerate}
\item
\begin{enumerate}[(a)]
\item if $\nabla_G\text{-}\deg(\nabla\Phi, \Omega)\neq\Theta\in U(G)$, then $(\nabla\Phi)^{-1}(0)\cap\Omega\neq\emptyset$,
\item if $\Omega=\Omega_1\cup\Omega_2$ and $\Omega_1,\ \Omega_2$ are open, disjoint  and $G$-invariant sets, then
\[\nabla_G\text{-}\deg(\nabla\Phi, \Omega)=\nabla_G\text{-}\deg(\nabla\Phi, \Omega_1)+\nabla_G\text{-}\deg(\nabla\Phi, \Omega_2),\]
\item if $\Omega_1\subset\Omega$ is an open and $G$-invariant set and  $(\nabla\Phi)^{-1}(0)\cap\Omega\subset\Omega_1$, then \[\nabla_G\text{-}\deg(\nabla\Phi, \Omega)=\nabla_G\text{-}\deg(\nabla\Phi, \Omega_1),\]
\item if $0\in\Omega$ and $\Phi\in C^2_G(\Omega,\R)$ is such that $\nabla\Phi(0)=0$ and $\nabla^2\Phi(0)\colon\H\to\H$ is a $G$-equivariant self-adjoint isomorphism then there is $\gamma_0>0$ such that for every $\gamma<\gamma_0$ we have
    \[\nabla_G\text{-}\deg(\nabla\Phi, B_{\gamma}(\H))=\nabla_G\text{-}\deg(\nabla\Phi^2(0), B(\H)).\]
\end{enumerate}
\item Fix $\Phi\in C^1_G(\H\times[0,1],\R)$ such that $(\nabla_u\Phi)^{-1}(0)\cap(\partial\Omega\times[0,1])=\emptyset$ and
$\nabla_u\Phi(u,t)=Lu-\nabla_u\eta(u,t)$, where $\nabla_u\eta\colon\Omega\times[0,1]\to\H$ is $G$-equivariant and compact. Then
\[\nabla_G\text{-}\deg(\nabla_u\Phi(\cdot,0), \Omega)=\nabla_G\text{-}\deg(\nabla_u\Phi(\cdot,1), \Omega).\]
\item Let $\Omega_1\subset\H_1,\ \Omega_2\subset\H_2$ be open, bounded and $G$-invariant subsets of $G$-representations $\H_1,\ \H_2$. Assume that the functionals $\Phi_i\in C^1_G(\H_i,\R),\ i=1,2$ are of the form
    $\Phi_i(u)=\frac{1}{2}\langle L_i u, u\rangle +\eta_i(u)$ and satisfy the assumptions (a1)-(a4). Define a functional $\Phi\in C^1_G(\H_1\oplus\H_2,\R)$ by $\Phi(u_1,u_2)=\Phi(u_1)+\Phi(u_2)$ and set $\Omega=\Omega_1\times\Omega_2$. Then
    \[\nabla_G\text{-}\deg(\nabla\Phi, \Omega)=\nabla_G\text{-}\deg(\nabla\Phi_1, \Omega_1)\star\nabla_G\text{-}\deg(\nabla\Phi_2, \Omega_2).\]
\end{enumerate}
\end{Theorem}

\begin{Theorem}\label{GLOB}
Fix $\Phi\in C^2_G(\H\times\Lambda,\R)$ such that $\nabla_u\Phi(u,\lambda)=Lu-\nabla_u\eta(u,\lambda)$, where the mapping $\nabla_u\eta\colon\Omega\times\Lambda\to\H$ is $G$-equivariant and compact. Suppose that $\nabla_u\Phi(0,\lambda)=0$ for every $\lambda\in\Lambda$. If there exist $\gamma_1,\ \gamma_2>0$ such that
\[\nabla_G\text{-}\deg(\nabla_u\Phi(\cdot,\lambda_1), B_{\gamma_1}(\H))\neq\nabla_G\text{-}\deg(\nabla_u\Phi(\cdot,\lambda_2), B_{\gamma}(\H)),\]
then at every path joining $(0,\lambda_1)$ and $(0,\lambda_2)$ exists a global bifurcation point of solutions of the equation $\nabla_u\Phi(u,\lambda)=0$.
\end{Theorem}

See \cite{degree} for properties of the degree and \cite{Geba, Rybicki} for the definition of the degree for gradient $G$-equivariant maps. For the general theory of the equivariant degree we refer the reader to \cite{BKS}, \cite{BKR}.

\end{document}